\documentclass{amsart}[12pt]

\usepackage{amsmath,amssymb,amsfonts}
\usepackage{graphicx}

\newtheorem{theorem}{Theorem}[section]
\newtheorem{proposition}[theorem]{Proposition}
\newtheorem{lemma}[theorem]{Lemma}

\theoremstyle{definition}
\newtheorem{definition}[theorem]{Definition}

\theoremstyle{remark}
\newtheorem{remark}[theorem]{Remark}

\numberwithin{equation}{section}

\newcommand{\inner}[2]{\langle#1,#2\rangle}

\newcommand{\CC}{{\mathbb C}}
\newcommand{\LL}{L}
\newcommand{\RR}{{\mathbb R}}
\newcommand{\ZZ}{{\mathbb Z}}

\newcommand{\calH}{{\mathcal H}}
\renewcommand{\L}{{\mathcal L}}
\renewcommand{\O}{{\mathcal O}}

\newcommand{\fg}{{\mathfrak g}}

\DeclareMathOperator{\curv}{curv}
\DeclareMathOperator{\hol}{hol}
\DeclareMathOperator{\Int}{Int}
\DeclareMathOperator{\Lie}{Lie}
\DeclareMathOperator{\Log}{Log}
\DeclareMathOperator{\supp}{supp}

\renewcommand{\Log}{\log}

\newcommand{\frakt}{{\mathfrak t}}
\newcommand{\frakg}{{\mathfrak g}}
\newcommand{\frakk}{{\mathfrak k}}

\newcommand{\bbR}{{\mathbb R}}
\newcommand{\bbZ}{{\mathbb Z}}

\newcommand\Cinf{\mathcal{C}^\infty}
\newcommand{\st}[1]{\ensuremath{^{\scriptstyle \textrm{#1}}}}

\begin{document}


\newcommand{\Ham}{\mbox{\sc Ham}}

\newcommand{\Aut}{\mbox{\sc Aut}}

\title{The spectral density function of a toric variety}

\author{D. Burns}\address{Department of Mathematics\\
University of Michigan \\Ann Arbor, MI 48109}
\thanks{D.B. supported in part by NSF grant DMS-0514070.}
\email{dburns@umich.edu}
\author{V. Guillemin}\address{Department of Mathematics\\
Massachusetts Institute of Technology\\ Cambridge, MA 02139} \email{vwg@mit.edu}
\thanks{V.G. supported in part by NSF grant DMS-0408993.}
\author{A. Uribe}
\thanks{A.U. supported in part by NSF grant DMS-0401064.}
\address{Department of Mathematics\\
University of Michigan\\Ann Arbor, MI 48109}
\email{uribe@umich.edu}

\maketitle

\tableofcontents

\begin{abstract}
For a K\"ahler manifold $X, \omega$ with a holomorphic line bundle $L$ and metric $h$ 
such that the Chern form of $L$ is $\omega$, the spectral measures are the measures 
$\mu_N = \sum |s_{N,i}|^2 \nu$, where $\{s_{N,i}\}_i$ is an 
$L^2$-orthonormal basis for $H^0(X, L^{\otimes N})$, and $\nu$ is Liouville measure. We study the asymptotics in $N$ of $\mu_N$ for $X, L$ a {\em Hamiltonian} toric manifold, 
and give a very precise expansion in terms of powers $1/N^j$ and data on the moment
polytope $\Delta$ of the Hamiltonian torus $K$ acting on $X$.   In addition, for a character $k$ of $K$ and the unique unit eigensection $s_{Nk}$ for the character $N k$ of the torus action on $H^0(X, L^{\otimes N})$, we give a similar expansion for the measures $\mu_{Nk} = |s_{Nk}|^2 \nu$.  A final remark shows that the eigenbasis  $\{s_{k}, k \in \Delta \cap \mathbb{Z}^{\dim K} \}$ is a Bohr-Sommerfeld basis in the sense of \cite{Ty}, 
and that the asymptotic results of \cite{BPU} are exact in this case. 

Some of the present results are closely related to earlier results of \cite{10}. The present paper uses no microlocal analysis, but rather an Euler-Maclaurin formula for Delzant polytopes.
\end{abstract}

\section{Introduction}
\label{sec:1}

The purpose of this note is to explore a fundamental problem
in spectral theory in the context of ``toric geometry".
This problem, formulated in the context of Riemannian
geometry, is the following:  Let $M$ be a compact Riemannian manifold, and let
$\varphi_i$, $i=1,\,2,\ldots$ be an orthonormal basis of
eigenfunctions of the Laplace operator.  What can one say
about  the spectral measures
\begin{equation}\label{new1.1}
\mu_i = |\varphi_i|^2\,dx
\end{equation}
as $i$ tends to infinity?  For instance, if the geodesic flow
on $T^*M$ is ergodic, Schnirelman-Colin de Verdi\`ere-Zelditch
proved that along ``most" subsequences $i_1,\ i_2,\cdots$,
$\mu_i$ tends weakly to the volume measure, $dx$.
(This phenomenon is known as ``quantum ergodicity", and its violation
by certain exceptional sequences of $\varphi_i$'s as ``quantum scarring".)
However, what can one say about the limiting behavior of the $\mu_i$'s
if one makes other assumptions about geodesic flow, e.\ g.\ that it be
periodic or completely integrable?  If geodesic flow is periodic then it
is known that for each geodesic $\gamma$ there exists a sequence of
``quasi-modes", $\phi_i$, $i=i_1,\ i_2,\ldots$, such that $\mu_i$
tends in the limit to a delta function on $\gamma$.  On the other
hand it is also known that the eigenvalues of $\sqrt{\Delta}$ clump
into clusters, 
\[
\lambda_{i,k},\quad k=1,\ldots,N_i
\]
with $|\lambda_{i,k}-(ai+b)| = O(i^{-1})$ 
(for suitable constants $a$, $b$), and it is known that a 
vestige of quantum ergodicity survives:  The measures
\begin{equation}\label{new1.2}
\nu_i = \sum_k \mu_{i,k} = \sum_k |\varphi_{i,k}|^2\,dx
\end{equation}
tend in limit to the volume measure.  (The simplest example is 
$S^{n-1}$.  In this case $\nu_i$ is SO$(n)$ invariant and hence
\underline{is} the volume measure up to a constant factor.)

\medskip
There is another important instance in which the eigenvalues can be clumped
into clusters:  If a compact Lie group $K$ acts on $M$ by isometries
one can decompose the eigenspaces of $\Delta$ into $K$-invariant subspaces,
and consider the spectral measures (\ref{new1.2}), where the
$\varphi_{i,k}$'s are orthonormal bases of these subspaces.
If $K$ is an $n$-torus it is also natural to study the asymptotic
behavior of the measure (\ref{new1.1}) not for arbitrary
sequences of $\varphi_i$s, but for sequences for which $\varphi_i$ lies
in a weight space of $K$ of weight $\alpha_i$ and the $\alpha_i$
tend asymptotically to infinity along a ray in $\frakk^*$.  In both these
cases one would like to be able to relate the asymptotics of
$\mu_i$ and $\nu_i$ to properties of the geodesic flow.

\medskip
These problems have analogues in K\"ahler geometry:  If $X$ is a compact
K\"ahler manifold and $L\to X$ a Hermitian line bundle whose curvature form
is the negative of the K\"ahler form, then one can consider the asymptotic
behavior of the measures
\[
\mu_N = \sum |\varphi_{N,k}|^2\,\nu
\]
where $\{\varphi_{N,k}\;;\;k=1,\ldots d_N\}$ is an orthonormal basis of
$\Gamma_{\mbox{\tiny hol}} (L^N)$ (holomorphic sections of $(\LL^N)$) 
and $\nu$ is Liouville measure.
\footnote{This measure can be defined intrinsically as the measure
\[
C^\infty(X)\ni f\mapsto \mbox{Trace}\ \Pi_N\, M_f\,\Pi_N,
\]
where $\Pi_N$ is the orthogonal projection of
$\Gamma (L^N)$ onto
$\Gamma_{\mbox{\tiny hol}} (L^N)$, and $M_f$ is the 
operator ``multiplication by $f$".}
(Notice that $\mu_N$ is now a measure on ``phase space".
The analogue of $X$ in the case of periodic geodesic flow is the
quotient of the unit cotangent bundle of $M$ by the flow.)
Using general results about the microlocal structure of
Szeg\"o kernels, \cite{2}, one can
prove that the $\mu_N$ have a weak asymptotic expansion
as $N\to\infty$ with leading term Liouville measure.

If there is an action on $X$ of a torus, $K$,
preserving the K\"ahler structure and preserving $L$,
one can decompose the spaces $\Gamma_{\mbox{\tiny hol}} (L^N)$
into weight spaces and, as above, study the asymptotics of
the measure $|\varphi_{N,k}|^2\,\nu$ associated with
sequences of weights which tend asymptotically to infinity along
rays in $\frakk^*$.

\medskip
The purpose of this article is to examine both of these problems 
in the setting of ``toric geometry".  As a toric variety 
(together with its canonical K\"ahler metric) is completely
determined by its moment polytope, it is natural to seek results
formulated explicitly in terms of the polytope.  

\newcommand{\ND}{[N\Delta]}
\medskip
In more detail, let $K$
be an $n$-dimensional torus, $X$ a (non-singular)
$K$-toric variety, $\Phi : X
\to k^*$ the moment map associated with the action of $K$ on $X$
and $\Delta = \Phi (X)$ the moment polytope.  Under the action of
$K$ the space, $\Gamma_{\hol} (\LL^N)$, 
breaks up into an
orthogonal direct sum of one-dimensional weight spaces
\begin{displaymath}
  \Gamma_{\hol} = \bigoplus_{k\in \ND} \Gamma_{k}
\end{displaymath}
indexed by the set $[N\Delta]$ of integer lattice points
in the dilated polytope $N\Delta$, and thus
\begin{equation}
  \label{eq:1.1}
  \mu_N = \sum \langle s_{k}, s_{k}\rangle\, \nu,
\end{equation}
where $\{s_{k} \in \Gamma_{k}\;;\;k \in \ND\}$ is
an orthonormal basis of $\Gamma_{\hol}$ 
and $\langle s_{k} , s_{k}\rangle (p)$ is the
norm-squared of $s_{k} (p) \in \LL^N_p$.  Thus to understand
the asymptotic behavior of $\mu_N$ one has to understand the
asymptotic behavior of the functions $\langle s_{k},
s_{k}\rangle$.    Our first step in this direction is the
following explicit formula for this function.  Let $d$ be the
number of facets of the polytope $\Delta$, and $\ell_i : \Delta
\to \RR $, $i=1,\ldots ,d$, the lattice distance to the $i$\st{th}
facet (see  definition \ref{Defdistance}).  Then
 \begin{equation}
   \label{eq:1.2}
     \langle s_{k}, s_{k} \rangle = \frac{1}{c_{k}}
       \left( \phi^* \exp \left(N \sum_{i=1}^d \ell_i \left(\frac{k}{N}\right)
         \Log \ell_i - \ell_i \right) \right)
 \end{equation}
where $c_{k}$ is the integral of the expression in
parentheses.

The measure, $\mu_N$, is $K$-invariant, so it is completely
determined by its push-forward to $X/K$.  Moreover, $\Phi$ is
also $K$-invariant, so it defines a map $X/K \to \Delta$ which
for toric varieties is a bijection.  Hence to study the
asymptotics of $\mu_N$ it suffices to study the asymptotics of
the measure
\begin{displaymath}
  \mu^{\sharp}_N =: \Phi_* \mu_N \, .
\end{displaymath}
Moreover, for toric varieties $\Phi_* \nu$ is just ordinary
Lebesgue measure on $\Delta$, hence by (\ref{eq:1.2})
$\mu^{\sharp}_N$ is the measure
\begin{equation}
  \label{eq:1.3}
  \mu^{\sharp}_N = 
  \sum_{k \in \ND} \frac{1}{c_{k}} \exp \left(
    N \sum_{i=1}^d \ell_i \left( \frac{k}{N}\right) \Log \ell_i
    -\ell_i \right) \, dx \, .
\end{equation}
For $x$ and $y$ in $\Delta$ and $N\in \ZZ_+$ let
\begin{eqnarray}
  \label{eq:1.4}
  K_N (x,y) &=& c_N (x)^{-1} \exp \left( N \sum_{i=1}^d \ell_i (x) \Log
    \ell_i (y) - \ell_i (y) \right)\\
\noalign{\hbox{where}} \nonumber\\
  \label{eq:1.5}
c_N (x) &=& \int_{\Delta} \exp \left( N \sum_{i=1}^d \ell_i (x) \Log
     \ell_i (y) - \ell_i (y) \right) \, dy \, .
\end{eqnarray}
Then, for $f \in \Cinf (\Delta )$,
\begin{eqnarray}
  \label{eq:1.6}
  \int_{\Delta} f \, d\mu^{\sharp}_N &=& \sum_{k \in \ND}
    f^{\sharp}_N \left( \frac{k}{N} \right)\\
\noalign{\hbox{where}}\nonumber\\ 
  \label{eq:1.7}
f^{\sharp}_N (x) &=& \int_{\Delta} K_N (x,y) f(y) \, dy \, .
\end{eqnarray}
One of the main result of this paper is the following.

\begin{theorem}\label{th:1.1}
  There exist, for $i=0,1,2,\ldots ,$ differential operators, $P_i
  (x,D) : \Cinf (\Delta ) \to \Cinf (\Delta )$, of order $2i$
  with the property
  \begin{equation}
    \label{eq:1.8}
    \int_{\Delta} K_N (x,y) f(y) \, dy \sim \sum_{i=0}^\infty P_i (x,D) f\,  N^{-i}\, . 
  \end{equation}
Moreover, $P_0 =I$.
\end{theorem}

The $P_i$'s are \emph{combinatorial} invariants of the
polytope $\Delta$ (albeit given by rather
complicated formulas).  By combining this result with an
Euler--Maclaurin formula for Riemann sums over polytopes (see
\cite{8}), we will be able to write the sum (\ref{eq:1.3}) as an
asymptotic series in inverse powers of $N$ in which the
individual terms are integrals over the faces of $\Delta$ of
differential expressions in~$f$.  Moreover, if
$k$ is a lattice point of $\Delta$ and $\mu_{N,k}$ is the measure
\[
\mu_{Nk} = \inner{s_{Nk}}{s_{Nk}}\,\nu
\]
the formula (\ref{eq:1.8}) yields as a corollary a second main result
of this paper:

\begin{theorem}
For $f\in C^\infty(\Delta)$ one has an asymptotic expansion
\[
\int_X\phi^*f\ d\mu_{Nk} \sim 
\Bigl(\sum_{i=0}^\infty P_i(x,D)f\ N^{-i}\Bigr)|_{x=k}
\]
where the $P_i$ are the same operators as before.
\end{theorem}

In case $k$ is in the interior of $\Delta$
this result follows from the results in \S 7 and the
``matrix coefficients" estimates in \cite{BPU}.
However, we will give below a direct proof
that includes the case $k\in\partial\Delta$.

\medskip
To summarize briefly the contents of this article:  In
\S\ref{sec:2} we will review basic facts about toric varieties,
in \S\ref{sec:3} derive the formula (\ref{eq:1.2}), in
\S\ref{sec:4} prove Theorem~\ref{th:1.1} and in \S\ref{sec:5} derive
from it the asymptotic expansion mentioned above.
The asymptotic properties of $s_{k}$ that we discuss in
\S\ref{sec:4} are closely related to some results of
Shiffman--Tate--Zelditch, and can be viewed as an alternative derivation
of these results.  (See \cite{10}).  We will comment on the
relation of our work to theirs in \S6.  Also, on the open set where
$\langle s_{k}, s_{k} \rangle$ is non-zero, $-\Log
\langle s_{k}, s_{k} \rangle$ is a potential for the
K\"ahler metric on $X$, so inter alia our results give a formula
for this K\"ahler potential in terms of moment polytope data.
(For other formulas of this type see \cite{3}, \cite{4} and \cite{6}.)
Finally, in \S  7 we show that the basis of sections $\{s_k\}$
is a ``Bohr-Sommerfeld basis" in the sense of Tyurin, \cite{Ty}.

It may be worth noting that we
make no use of microlocal analysis in this paper.

\section{Toric varieties}
\label{sec:2}

Let $T$ be the standard $d$-dimensional torus, $T=(S^1)^d$, let
$\frakt=\Lie T =\RR^d$ and let $e_1,\ldots ,e_d$ be the standard basis
vectors of $\RR^d$.  $T$ acts on $\CC^d$ by its diagonal action,
and if we equip $\CC^d$ with the K\"ahler form, $\omega =
\sqrt{-1}\,  \sum \, dz_i \wedge \, d\bar{z}_i$ this becomes a
Hamiltonian action with moment map
\begin{equation}
  \label{eq:2.1}
  \phi : \CC^d \to \frakt^* \, , \quad z \to \sum |z_i |^2 e^*_i \, .
\end{equation}
For $G$ a codimension $n$ subtorus of $T$, let $\fg = \Lie G$,
let $\ZZ^*_G \subset \fg^*$ be the weight lattice of $G$ and let
\begin{equation}
  \label{eq:2.2}
  L: \frakt^* \to \fg^*
\end{equation}
be the transpose of the inclusion map, $\fg \to \frakt$.  Then the
action of $T$ on $\CC^d$ restricts to a Hamiltonian action of $G$
on $\CC^d$ with moment map
\begin{equation}
  \label{eq:2.3}
  \psi = L \circ \phi = \sum |z_i |^2 \alpha_i
\end{equation}
where $\alpha_i = L e^*_i \in \ZZ^*_G$.  We will assume that this
moment map is proper, or alternatively, that the $\alpha_i$'s are
``polarized'':  for some $\xi \in \fg$, all the numbers, $\alpha_i
(\xi)$ are positive.  The toric varieties we will be considering
in this paper are symplectic reduced spaces of the form
\begin{equation}
  \label{eq:2.4}
  X_{\alpha} = Z_{\alpha}/G
\end{equation}
where $\alpha$ is in $\ZZ^*_G$ and $Z_{\alpha} =
\psi^{-1}(\alpha)$.  We recall that the action of $G$ on
$\psi^{-1}(\alpha)$ is locally free iff $\alpha$ is a regular
value of $\psi$, and since we will only be considering
non-singular toric varieties in this paper we will assume that
$G$ acts freely on $Z_{\alpha}$.  Hence $X_{\alpha}$ is a
manifold and the projection
\begin{equation}
  \label{eq:2.5}
  \pi : Z_{\alpha} \to X_{\alpha}
\end{equation}
is a principal $G$-fibration.  From the action of $T$ on $\CC^d$
we get a Hamiltonian action of $T$ on $X_{\alpha}$, and if we
denote by ``$\iota$'' the inclusion of $Z_{\alpha}$ into $\CC^d$,
the moment map for the $T$-action on $\CC^d$ is related to the
moment map for the $T$ action on $X_{\alpha}$ by the identity:
\begin{equation}
  \label{eq:2.6}
  \phi_{\alpha} \circ \pi = \phi \circ \iota \, .
\end{equation}
Thus we have a commutative diagram:
\[
\begin{array}{cccll}
Z_\alpha &\stackrel{\iota}{\hookrightarrow} &\CC^d  & & \\
{}_\pi\downarrow \ \ & &{}_\phi\downarrow\ \  & \searrow^{\psi} &\\
X_\alpha &\overset{\phi_\alpha}{\to}&\frakt^*&\stackrel{L}{\to} &\frakg^*
\end{array}
\]
The moment polytope for the action of $T$ on
$X_{\alpha}$ is
\begin{equation}
  \label{eq:2.7}
  \Delta_{\alpha} = \RR^d \cap L^{-1} (\alpha),
\end{equation}
by (\ref{eq:2.1}) and (\ref{eq:2.3}).  (Here we've identified
$\frakt^*$ with $\RR^d$ via the basis vectors, $e^*_i$.)  The facets
of this polytope are the intersections of  $L^{-1}(\alpha)$ with
the coordinate hyperplanes $x_i=0$ in $\RR^d$.

\begin{definition}\label{Defdistance}
The ``lattice distance"
to the $i$\st{th} facet, $\ell_i$, is
the restriction of the coordinate function $x_i$ to $\Delta_{\alpha}$.
\end{definition}

Since $G$ acts trivially on $X_{\alpha}$ the action of $T$ on
$X_{\alpha}$ is effectively an action of the quotient group,
$K=T/G$, and since $\Lie{K} =\frakk=\frakt/\fg$, 
the dual $\frakk^*$ is the annihilator in $\frakt^*$
of $\fg$, and hence is the kernel of the map $L$.
To make the action of $K$ a Hamiltonian action one has to normalize
the moment map, $\phi_{\alpha}:X_{\alpha} \to L^{-1} (\alpha )$,
so that it maps into $\frakk^*$, and this one can do by fixing an
element, $c_{\alpha} \in \ZZ^d \cap \Int \Delta_{\alpha}$ and
replacing $\phi_{\alpha}$ by $\phi_{\alpha} - c_{\alpha}$.  We
won't, however, bother to make this normalization here and will
continue to think of $\phi_{\alpha}$ as a map into $L^{-1}(\alpha)$.

\section{K\"ahler reduction}
\label{sec:3}

In this section we review some general facts about K\"ahler
reduction and use them to derive the formula~(\ref{eq:1.2}).  Let
$M$ be a complex manifold and $\LL \to M$  a holomorphic line
bundle.  We recall that if $\LL$ is equipped with a Hamiltonian
inner product, there is a unique holomorphic connection on $\LL$
which is compatible with this inner product.  More explicitly, if
$\triangledown$ is a holomorphic connection and $\langle \, , \,
\rangle$ an inner product then for every holomorphic
trivialization $s: U \to \LL$
\begin{equation}
  \label{eq:3.1}
  \frac{\triangledown s}{s} = \mu \in \Omega^{1,0} (U)
\end{equation}
and the compatibility of  $\langle \, , \,
\rangle$ and $\triangledown$ reduces to
\begin{eqnarray}
  \label{eq:3.2}
d \Log  \langle s,s\rangle &=& \mu + \overline{\mu}\\
\noalign{\hbox{and hence}}\\ \notag
    \label{eq:3.3}
\partial \Log  \langle s,s\rangle &=& \mu
\end{eqnarray}
which shows that the inner product determines the connection and
vice versa.  It also shows that
\begin{equation}
  \label{eq:3.4}
  \curv (\triangledown) = \sqrt{-1}\,  \partial \overline{\partial}
    \Log  \langle s,s\rangle = : -\omega \, .
\end{equation}
Suppose now that the form, $\omega$, is K\"ahler.  Let $G$ be an
$m$-dimensional torus, let $\tau : G \times M \to M$ be a
holomorphic action of $G$ on $M$ and let $\tau^{\sharp} :G \times
\LL \to \LL$ be an action of $G$ on $\LL$ by holomorphic line
bundle automorphisms which is compatible with $\tau$.  If
$\tau^{\sharp}$ preserves  $\langle \, , \, \rangle$ then by
(\ref{eq:3.2})---(3.3) it preserves $\triangledown$ and
$\omega$.  Moreover, by Kostant's formula there is an
intrinsically defined  moment map, $\Phi :M \to \fg^*$, such that
\begin{equation}
  \label{eq:3.5}
  L_{v} s = \triangledown_{v_M} s + i\langle \Phi , v \rangle s
\end{equation}
for all $s \in \Cinf (\LL)$ and $v \in \fg$.  In other words
the infinitesimal action of $G$ on $\Cinf (\LL)$ is completely
determined by $\Phi$ and $\langle \, , \, \rangle$.

Let's now describe what symplectic reduction looks like from
this K\"ahlerian perspective.  Given $\alpha \in \ZZ^*_G$, let
$Z_{\alpha} = \Phi^{-1} (\alpha)$.  Assuming that $G$ acts freely
on $Z_{\alpha}$, the reduced space, $X_{\alpha} = Z_{\alpha} /G$
is a $\Cinf$ manifold, and the projection, $\pi : Z_{\alpha} \to
X_{\alpha}$ is a principal $G$-fibration.  Let $\LL_{\alpha} \to
X_{\alpha}$ be the line bundle whose fiber at $p$ is the
(one-dimensional) space of sections
\begin{displaymath}
  s: \pi^{-1}(p) \to \LL
\end{displaymath}
which transforms under $G$ by the recipe
\begin{equation}
  \label{eq:3.6}
  \tau^{\sharp} (\exp v)^* s = e^{ i \alpha (v)} s
\end{equation}
or alternatively, by (\ref{eq:3.5}), are auto-parallel along
$\pi^{-1}(p)$.  For such a section, $\langle s,s \rangle$ is
constant along $\pi^{-1}(p)$, so the inner product, $\langle \, ,
\, \rangle$, induces an inner product,  $\langle \, ,
\, \rangle_{\alpha}$ on $\LL_{\alpha}$.  In terms of sections,
if $\Cinf (\LL)^{\alpha}$ is the space of global sections of
$\LL$ which transform by (\ref{eq:3.6}) and $\iota : Z_{\alpha}
\to M$ is inclusion
\begin{eqnarray}
  \label{eq:3.7}
  \iota^* \Cinf (\LL)^{\alpha} &=& \pi^* \Cinf (\LL_{\alpha})\\
\noalign{\hbox{ and if  $ \iota^* s = \pi^* s_{\alpha}$}}\nonumber\\
  \label{eq:3.8}
  \iota^*  \langle s , s \rangle, 
      &=& \pi^* \langle s_{\alpha} , s_{\alpha} \rangle_{\alpha} \, .
\end{eqnarray}

We now define a complex structure on $X_{\alpha}$ and make
$\LL_{\alpha} \to X_{\alpha}$ into a holomorphic line bundle by
requiring that
\begin{eqnarray}
  \label{eq:3.9}
  \pi^* \O_{\alpha}  &\subseteq & \iota^* \O\\
\noalign{and} \nonumber\\
  \label{eq:3.10}
\pi^* \L_{\alpha} & \subseteq & \iota^* \L
\end{eqnarray}
where $\O$ is the sheaf of holomorphic functions on $M$,
$\O_{\alpha}$ the sheaf of holomorphic functions on $X$, $\L$ the
sheaf of holomorphic sections of $\LL$ and
$\L_{\alpha}$ the sheaf of holomorphic sections of $\LL_{\alpha}$.
By (3.3) one gets a holomorphic connection,
$\triangledown_{\alpha}$, on $\LL_{\alpha}$ which is compatible
with $\langle \, , \, \rangle_{\alpha}$ and by (\ref{eq:3.4})
and (\ref{eq:3.8})
\begin{equation}
  \label{eq:3.11}
  \iota^* \curv (\triangledown) 
     = \pi^* \curv (\triangledown_{\alpha}) \, .
\end{equation}
Thus $-\curv  (\triangledown_{\alpha})$ is the reduced symplectic
form on $X_{\alpha}$.

We'll conclude this section by applying these general
observations to the set-up in \S\ref{sec:2}.

Let $\LL = \CC \times \CC^d \to \CC^d$ be the trivial line bundle
over $\CC^d$ and $s: \CC^d \to \LL$ the trivial section, $s(z) =
(z,1)$.  If we equip $\LL$ with the Hermitian inner product,
$\langle s,s \rangle = e^{-|z|^2}$, we get a non-trivial
connection on $\LL$, and by (\ref{eq:3.4}) the curvature form of
this connection is minus the symplectic form, $\sqrt{-1}\,  \sum
dz_i \wedge d\overline{z}_i$.  If we let $G$ act on $\LL$ by
requiring that $s$ be $G$-invariant then the space $\Gamma_{\hol}
(\LL)^{\alpha}$ is spanned by monomials, $z^{k_1}_1 \ldots
z^{k_d}_d$, for which $\alpha = L (\sum k_i e^*_i) = \sum k_i
\alpha_i$, i.e.,~for which $k \in \ZZ^d \cap
\Delta_{\alpha}$.  For each of these sections let $s_{k}$ be
the corresponding holomorphic section of $\LL_{\alpha}$.  On
$Z_{\alpha}$ one has
\begin{displaymath}
  \langle z^{k},z^{k}\rangle 
     = |z_1|^{2k_1} \ldots |z_d|^{2k_d} e^{-|z|^2},
\end{displaymath}
and so by (\ref{eq:3.8}) and  (\ref{eq:2.1})
\begin{displaymath}
  \pi^* \langle s_{k},s_{k}\rangle _{\alpha}
    = \iota^* |z_1 |^{2k_1} \cdots |z_d|^{2k_d}
        e^{-|z|^2} =  \iota^*\phi^* (x^{k_1}_1\ldots x^{k_d}_d e^{-\sum x_i}).
\end{displaymath}
But, by (\ref{eq:2.6})
\[
\iota^* \phi^* (x^{k_1}_1\ldots x^{k_d}_d  e^{-\sum x_i}) = 
\pi^* \phi^*_{\alpha}  (\ell^{k_1}_1\ldots \ell^{k_d}_d   e^{-\sum \ell_i}).
\]
%
Hence we conclude that
\begin{equation}
  \label{eq:3.12}
  \langle s_{k},s_{k}\rangle_\alpha = \phi^*_{\alpha} \Bigl(
     \ell^{k_1}_1\ldots \ell^{k_d}_d e^{-\sum \ell_i}\Bigr),
\end{equation}
which implies the formula (\ref{eq:1.2}).

\section{Asymptotics}
\label{sec:4}

In this section we will use stationary phase  to
analyze the behavior of the integral (\ref{eq:1.7}) as $N$ tends
to infinity.  Let
\begin{equation}
  \label{eq:4.1}
  \varphi (x,y) = \sum \ell_i (x) \Log \ell_i (y)-\ell_i(y)
\end{equation}
be the phase function in this integral.  We claim:

\begin{lemma}
  \label{lem:4.1}
For $x$ a fixed point in the interior of $\Delta$, the function $\varphi$, regarded as a function of
$y$, has a unique critical point at $x=y$, and this critical
point is the unique global maximum of the function, $\varphi$, on
$\Delta$.
\end{lemma}

\begin{proof}
Since
\begin{displaymath}
d\varphi = \sum \ell_i (x) \frac{d\ell_i}{\ell_i (y)}
    - d\ell_i =0
\end{displaymath}
at $x=y$, the point $x=y$ is a critical point of $\varphi$ and
since
\begin{equation}
  \label{eq:4.2}
  d^2\varphi = - \sum \ell_i (x) \frac{(d\ell_i)^2}{\ell_i (y)^2}
\end{equation}
this critical point is a maximum.  Moreover, by (\ref{eq:4.2})
every critical point of $\varphi$ in the interior of $\Delta$ has
to be a maximum and as $y$ tends to the boundary of $\Delta$,
$\varphi$ tends to $-\infty$.  Hence by the ``peaks--passes''
lemma $x$ is the only critical point of $\varphi$ and is its
unique global maximum.
\end{proof}

For $x$ in the interior of a boundary face, $F$, of $\Delta$,
one has an analogous result:

\begin{lemma}\label{lem:4.2}
The restriction of $\varphi$ to $F$ has a unique critical point at
$x=y$, and this critical point is the unique global maximum of
$\varphi$ on $\Delta$.  In addition, the derivatives 
of $\varphi$ (as a function of $y$) 
in directions normal to $F$ are not zero at $y=x$.
\end{lemma}
\begin{proof}
Suppose that $F$ is defined by the equations $\ell_i = 0$, 
$i\in I\subset\{1,\,2,\ldots ,d\}$.  Then, for $x$ in the interior of $F$,
\begin{equation}\label{eq:new4.2}
\varphi(x,y) = \sum_{i\not\in I} 
 \ell_i (x) \Log \ell_i (y)-\sum_{i=1}^d\ell_i(y), 
\end{equation}
which shows that $\varphi$ is a decreasing function of
the $\ell_i,\ i\in I$.    The joint minimum of those functions
is exactly $F$, and therefore the global maximum of 
$\Delta\ni y\mapsto \varphi(x,y)$ is attained on $F$.
The restriction of this function to $F$ is
\[
\varphi(x,y)|_{y\in F} = \sum_{i\not\in I} 
 \ell_i (x) \Log \ell_i (y)-\sum_{i\not\in I} \ell_i.
\]
Lemma \ref{lem:4.1} can be applied to this restriction,
and therefore $y=x$ is the unique global maximum of $\varphi|_F$,
and so $x$ is the unique global maximum of $\Delta\ni y\mapsto \varphi(x,y)$
on $\Delta$.  Moreover, at $y=x$
\begin{equation}\label{eq:4diffvarphi}
(d\varphi)_y = -\sum_{i\in I} (d\ell_i)_y,
\end{equation}
which shows that the derivatives normal to $F$ 
are not zero at $x$ (and $\varphi$ decreases to the interior
of the polytope).  

\end{proof}

From these lemmas one obtains the following ``localization''
theorem for the integral operator defined by~(\ref{eq:1.8}).

\begin{theorem}
  \label{th:4.2}
  Let $f$ and $g$ be in $\Cinf (\Delta)$.  Suppose that for $x
  \in \Delta$, $f(x) \ne 0$ and $x \notin \supp g$.  Then
  \begin{equation}
    \label{eq:4.3}
    \left| \int_{\Delta} e^{N \varphi (x,y)} g(y) \, dy \right|
      \leq e^{-cN} \left| \int_{\Delta} e^{N\varphi (x,y)}
        f (y) \, dy \right|
  \end{equation}
for some positive constant, $c$.

\end{theorem}

We will now examine the local behavior of the transform
(\ref{eq:1.7}) in the neighborhood of a fixed vertex, $p$, of
$\Delta$.  Let $\Delta_p$ be the open subset of $\Delta$ obtained
by deleting from $\Delta$ all facets except the facets containing
$p$.  By repagination we can assume that these are the facets,
$\ell_i =0$, $i=1,\ldots ,n$.

\begin{lemma}
  \label{lem:4.3}

There exists an affine transform mapping $\Delta_p$ onto an open
subset of the positive orthant, $\RR^n_+$, mapping $p$ onto the
origin and transforming the $\ell_i$'s, $i=1,\ldots ,n$, into the
coordinate functions, $x_i$, $i=1,\ldots ,n$.

\end{lemma}

(For proof of this ``standard fact'' about moment polytopes of
toric manifolds see \cite{5}.)

In these new coordinates the phase function (\ref{eq:4.1}) takes
the form
\begin{eqnarray}
  \label{eq:4.4}
  \varphi (x,y) &=& \sum x_i \log y_i - y_i + \psi (x,y)\\
\noalign{\hbox{where}\nonumber}\\
\label{eq:4.5}
\psi (x,y) &=& \sum_{r>n} \ell_r (x) \log \ell_r (y) -\ell_r (y)
\end{eqnarray}
is a $\Cinf$ function on $\Delta_p$.  Moreover, with $x$ fixed,
the derivative of $\psi$ with respect to $y$:
\begin{displaymath}
  d\psi = \sum_{r>n} \frac{\ell_r (x)}{\ell_r (y)}
    \, d\ell_r - \, d\ell_r
\end{displaymath}
is zero at $x=y$, so
\begin{eqnarray}
  \label{eq:4.6}
  \frac{\partial \psi}{\partial y_i} (x,y)  
    &=&    \sum h_{i,j} (x,y) (x_j -y_j)\\[1ex]
\noalign{\hbox{and}}\nonumber\\
   \frac{\partial \varphi}{\partial y_i} &=& \frac{x_i - y_i}{y_i}
      +\sum_j h_{i,j} (x,y) (x_j-y_j)\nonumber\\[1ex]
      &=& \frac{1}{y_i} \sum_j (\delta_{i,j} +y_i h_{i,j})
         (x_j-y_j) \, .\nonumber
 \end{eqnarray}
Hence
\begin{equation}
  \label{eq:4.7}
  x_j-y_j = \sum g_{i,j} (x,y) y_i 
        \frac{\partial\varphi}{\partial y_i}
\end{equation}
the $g_{i,j}$'s being $\Cinf$ in a neighborhood of $x=y=0$.

Consider  now an integral of the form
\begin{equation}
  \label{eq:4.8}
  \int_{\RR^n_+} e^{N\varphi (x,y)} f(x,y)\, dy
\end{equation}
where $f$ is $\Cinf$ and supported in a neighborhood of $x=y=0$.
Let $\rho (y)$ be a $\Cinf_0$ function which is equal to one on a
neighborhood of the support of $f$.  Then
\begin{eqnarray}
  \label{eq:4.9}
  f(x,y) &=& \left(f_0 (x) + \sum (y_j - x_j) f^{\sharp}_j (x,y)\right)
     \rho (y) \, , \\
\noalign{\hbox{where}}\nonumber\\
\label{eq:4.10}
f_0 (x) &=& f(x,x)\\
\noalign{\hbox{and}}\nonumber\\
\label{eq:4.11}
f^{\sharp}_j (x,y) &=&  \int^1_0 \frac{\partial}{\partial y_j}
     f (x,x + t (y-x)) \, dt \, .
\end{eqnarray}
By (\ref{eq:4.7}) and (\ref{eq:4.9}) we can write (\ref{eq:4.8})
as the sum of the two expressions
\begin{eqnarray}
  \label{eq:4.12}
  f_0 (x) \int_{\RR^n_+} e^{N \varphi (x,y)}\rho (y) \, dy\\
\noalign{\hbox{and}}\nonumber\\
\label{eq:4.13}
     -\int_{\RR^n_+} e^{N \varphi (x,y)}
         \left(\sum_{i,j}g_{i,j}f^{\sharp}_j y_i 
            \frac{\partial  \varphi}{\partial y_i}\right)
               \rho (y) \, dy
\end{eqnarray}
and by making the substitution
\begin{displaymath}
  e^{N \varphi (x,y)} \frac{\partial \varphi}{\partial y_i}
    = \frac{1}{N} \, \frac{\partial}{\partial y_i}
       e^{N \varphi (x,y)}
\end{displaymath}
and integrating by parts with respect to $y_i$ we can rewrite
(\ref{eq:4.13}) in the form
\begin{eqnarray}
  \label{eq:4.14}
  \frac{1}{N} \int_{\RR^n_+} e^{N\varphi (x,y)}
     f_1 (x,y) \, dy \\[1ex]
\noalign{\hbox{where}}\nonumber\\
\label{eq:4.15}
   f_1 (x,y) = \sum_{i,j} \, \frac{\partial}{\partial y_i}
      (y_i f^{\sharp}_j g_{i,j} \rho (y)) \, .
\end{eqnarray}
(Notice that in integrating by parts we don't pick up boundary
terms because of the presence of the $y_i$'s in the integrand.)

From (\ref{eq:4.12}) and (\ref{eq:4.14}) we get for
(\ref{eq:4.8}) the expansion
\begin{equation}
  \label{eq:4.16}
  f_0 (x) \int_{\RR^n_+} e^{N \varphi (x,y)} \rho (y) \, dy
    + \frac{1}{N} \int_{\RR} e^{N \varphi (x,y)} f_1 (x,y)\, dy
\end{equation}
and by iteration of (\ref{eq:4.16}), an expansion
\begin{eqnarray}
  \label{eq:4.17}
  \sum^{k-1}_{i=0} f_i (x) N^{-i} \int_{\RR^n_+}
     e^{N \varphi (x,y)} \rho (y) \, dy + R_k (x)\\
\noalign{\hbox{where}}\nonumber\\
\label{eq:4.18}
  R_k (x) = N^{-k} \int_{\RR^n_+} e^{N \varphi (x,y)}
     f_k (x,y) \, dy \, .
\end{eqnarray}
(In more detail:  $f_i (x) = f_i (x,x)$ and $f_i (x,y)$ is obtained from
$f(x,y)$ by iterating $i$~times the operation (\ref{eq:4.15}).
In particular $f_i$ is a sum of derivatives of $f$ of degree less
than or equal to $2i$ with $\Cinf$ functions as coefficients.)

Finally observe that for $x$ near zero the quotient of 
\begin{displaymath}
  \int_{\Delta} e^{N \varphi (xy)} (1-\rho (y)) \, dy
\end{displaymath}
by
\begin{equation}
  \label{eq:4.19}
  \int_{\Delta} e^{N \varphi (x,y)}\, dy
\end{equation}
is of order $O (e^{-cN})$ by Theorem~\ref{th:4.2}, hence if we
divide the sum (\ref{eq:4.17}) by (\ref{eq:4.19}) and let $k$ tend
to infinity we get the asymptotic expansion (\ref{eq:1.8}).

\begin{remark}
  \label{rem:4.4}

  In the discussion above we've assumed that $f(x,y)$ is
  supported on the set, $x,y \in \Delta_p$, however, by the
  localization Theorem~\ref{th:4.2} one can always reduce to this
  case by means of a partition of unity.

\end{remark}

\section{Riemann sums}
\label{sec:5}

One of the many variants of the classical Euler--Maclaurin
formula asserts that for $f \in \Cinf_0 (\RR)$ the Riemann sum
\begin{displaymath}
  \frac{1}{N} \sum^{\infty}_{k=0} f \left( -\frac{k}{N}\right)
\end{displaymath}
differs from the Riemann integral
\begin{displaymath}
  \int^0_{-\infty} f(x) \, dx
\end{displaymath}
by an asymptotic series
\begin{equation}
  \label{eq:5.1}
  \frac{f(0)}{2N} + \sum^{\infty}_{n=1} (-1)^{n-1}
    \frac{B_n}{(2n)!} f^{(2n-1)}(0) N^{-2n}
\end{equation}
where the $B_n$'s are the Bernoulli numbers.

Recalling that
\begin{displaymath}
  \tau (s) = : \frac{s}{1-e^{-s}} = 1+ \frac{s}{2} + \sum
     (-1)^{n-1} B_n \frac{s^{2n}}{(2n)!}
\end{displaymath}
this asymptotic expansion can be written more succinctly in the form:
\begin{equation}
  \label{eq:5.2}
  \frac{1}{N} \sum^{\infty}_{k=0} f \left( -\frac{k}{N}\right)
  \sim \left( \tau \left( \frac{1}{N}\, 
      \frac{\partial}{\partial h} \right)
      \int^h_{-\infty} f (x) \, dx \right) (h=0)\, .
\end{equation}

Guillemin and Sternberg have recently announced in \cite{8} an
$n$-dimensional version of this result in which the interval,
$(-\infty ,0]$, gets replaced by a convex polytope.  In
particular for the moment polytopes associated with toric
manifolds their formula is basically a ``product'' version of the
formula above and is proved by localization arguments similar to
those we  used above to prove Theorem~\ref{th:1.1}.  Let $\Delta
\subseteq \RR^n$ be such a polytope and let $d$ be the number of
facets of $\Delta$.  Then $\Delta$ can be defined by a set of
inequalities
\begin{equation}
  \label{eq:5.3}
  \langle u_i , x \rangle \leq c_i
\end{equation}
where $c_i$ is an integer and $u_i \in (\ZZ^n)^*$ is a primitive
lattice vector which is perpendicular to the $i$\st{th} facet and
points ``outward'' from $\Delta$.  The Euler--Maclaurin formula in
\cite{8} asserts:

\begin{theorem}
  \label{th:5.1}

Let $\Delta_h$ be the polytope 
\begin{equation}
  \label{eq:5.4}
  \langle u_i ,x \rangle \leq c_i + h_i \, , \quad
     i=1,\ldots , d \, .
\end{equation}
Then for $f \in \Cinf_0 (\RR^n)$
\begin{equation}
  \label{eq:5.5}
  \frac{1}{N^n} \sum_{k \in \ZZ^n \cap N\Delta} f
  \left( \frac{k}{N} \right) \sim \left( \tau \left(
      \frac{1}{N}\, \frac{\partial}{\partial h} \right)
    \int_{\Delta_h} f(x) \, dx \right) (h=0)  
\end{equation}
where $\tau (s_1 ,\ldots ,s_d) = \tau (s_1) \ldots \tau (s_d)$.

\end{theorem}

Now notice that if we divide (\ref{eq:1.6}) by $N^n$ the right
hand side is exactly a Riemann sum of the form above.  Hence if
we plug in for $f^{\sharp}_N$ the asymptotic expansion
(\ref{eq:1.8}) and apply to each summand the
formula~(\ref{eq:5.5}) we obtain an ``Euler--Maclaurin formula''
for the asymptotics of the measure $\mu_N$.

\section{The Shiffman--Tate--Zelditch results}
\label{sec:6}

Let $\varphi (x,y)$ be the function (\ref{eq:4.1}).  By
(\ref{eq:1.2})
\begin{equation}
  \label{eq:6.1}
  \langle s_k , s_k \rangle (p) = \frac{1}{c_k(x)}\; e^{N\varphi (x,y)}
\end{equation}
where $x= k/N$, $p \in \Phi^{-1} (y)$ and
\begin{equation}
  \label{eq:6.2}
  c_k(x) = \int_{\Delta} e^{N \varphi (x,y)} \, dy \, .
\end{equation}
If $x \in {\rm Int}\,  \Delta$ then by Lemma~\ref{lem:4.1} the
function
\begin{displaymath}
\Delta \ni  y  \mapsto \varphi (x,y)
\end{displaymath}
has a unique non-degenerate maximum at $y=x$, and hence by the
lemma of steepest descent
\begin{equation}
  \label{eq:6.3}
  c_k(x)  = \left( \frac{2 \pi}{N}\right)^{n/2} h(x)^{-\frac12}
      e^{N\varphi(x,x)} (1+0 (N^{-1}))
\end{equation}
where $h(x)$ is the determinant of the quadratic form
\begin{equation}
  \label{eq:6.4}
  \sum \frac{1}{\ell_i (x)} (d\ell_i)^2 (x) \, .
\end{equation}
Thus as $k=Nx$ tends to infinity along the ray through $x$ one
gets the asymptotic identity
\begin{equation}
  \label{eq:6.5}
  \langle s_k , s_k \rangle (p) \sim 
      \left( \frac{N} {2 \pi}\right)^{n/2} h(x)^{1/2}
         e^{N (\varphi (x,y)-\varphi (x,x))}
\end{equation}
at $p\in \Phi^{-1} (y)$.  In particular, as $N$ tends to infinity
$ \langle s_k , s_k \rangle$ concentrates exponentially on the
Bohr--Sommerfeld set, $\Phi^{-1} (k/N)$.  This result is due to
Shiffman, Tate and Zelditch,
\footnote{More or less:  Their result
  involves a slightly different choice of inner product on the
  $s_k$'s and of coordinates on $X$.  See \cite{10}.} 
who also observe that by applying steepest descent arguments
to the function (\ref{eq:new4.2}) one gets an analogue of (\ref{eq:6.5})
for $x$ lying in the interior of a face, $F$, of $\Delta$.  In this case the
asymptotic dependence of $\inner{s_k}{s_k}(p)$ on $N$ is given
by an expression similar to (\ref{eq:6.5}), except that the ``$n$"
in (\ref{eq:6.5}) has to be replaced by the dimension of $F$.
Hence the behavior of $\inner{s_k}{s_k}(p)$ for $k = Nx$ is
very non-uniform in $x$ when $x$ is near the boundary of $\Delta$.
We will prove below that by averaging their result over an ``$\delta$-pinched"
neighborhood
\[
\Bigl| \frac{k}{N}- y\Bigr| < \frac{1}{N^\delta},\qquad 0<\delta<\frac{1}{2}
\]
one gets a version of (\ref{eq:6.5}) which is much more uniform in $k/N$:

\begin{theorem}
For $x = k/N\in\Delta$, $\delta\in(0,1/2)$, and for every test function $\psi\in C_0^\infty(\bbR^n)$,
\begin{equation}\label{eq:6.7}
\int_\Delta \langle s_k\,,\,s_k\rangle\ 
\psi\Bigl( N^\delta \Bigl(\frac{k}{N}-y\Bigr)\Bigr)\ dy\  \sim\ 
\sum_{i=0}^\infty \sigma_i(x)\, N^{-(1+2\delta) i}\ ,
\end{equation}
the $\sigma_i(x)$ being $C^\infty$ functions on $\Delta$.
Thus the averaged estimate, unlike the pointwise estimate (\ref{eq:6.5}), is
``uniform up to the boundary".
\end{theorem}
\begin{proof}
We mimic the integration by parts argument in \S 4.  Applying this argument to the function
\begin{equation}\label{eq:6.6}
f(x,y) = \psi\Bigl( N^\delta \Bigl(x-y\Bigr)\Bigr)
\end{equation}
and keeping track of powers of $N$ one gets for the integral
\[
\int_{\bbR^n_{+}} e^{N\varphi(x,y)}\psi\Bigl( N^\delta \Bigl(x-y\Bigr)\Bigr)\ dy
\]
an expression
\[
\sum_{i=0}^{k-1} \sigma_i(x)\ N^{-i(1-2\delta)}\ \int_{\bbR^n_{+}} e^{N\varphi(x,y)}\ \rho(y)\ dy
+ R_k(x)
\]
where
\[
R_k(x) = N^{-k(1-2\delta)}\ \int_{\bbR^n_{+}} e^{N\varphi(x,y)}\ f_k(x,y,N^\delta(x-y))\ dy,
\]
and dividing by (\ref{eq:4.19}) and letting $k$ tend to infinity one gets the estimate
(\ref{eq:6.7}).
\end{proof}

\bigskip
Another result of \cite{10} which is closely related to the results of this paper
concerns the asymptotic behavior of another interesting measure associated
with the norm-squares of the $s_k$'s, namely the measure on the real line
\begin{equation}\label{eq:6.8}
\mu_N([t,\infty]) = \mbox{Vol }\{\,\inner{s_k}{s_k}(p) \geq t\,\}
\end{equation}
(i.~e.~ the distribution function of the ``random variable" $\inner{s_k}{s_k}$), where $k = Nx$.
Assuming that $x$ is a point in the interior of $\Delta$, Shiffman, Tate and 
Zelditch prove that the moments of this measure have the limiting behavior
\begin{equation}\label{eq:6.9}
\int_0^\infty t^m \ d\mu_N \sim (c\, N^{n/2})^{m-1}\ m^{-n/2},\qquad m=0, \ 1, \ldots
\end{equation}
where $c$ is a constant depending on $x$, and from this result deduce that 
$\mu_N$ satisfies ``universal rescaling laws" in various regimes
(e.\ g.\ for $t$ exponentially small with respect to $N$ or for $t$ greater than
some positive power of $N$).  To deduce (\ref{eq:6.9}) from the results above
we note that the integral on the left is just
\[
\int_X \inner{s_k}{s_k}^m\ d\nu
\]
where $\nu$ is, as in \S 1, Liouville measure.  This integral is equal to the
integral  over
$\Delta$ of the right-hand side of (\ref{eq:6.1}) to the $m$\st{th} power, with respect to Lebesgue measure.  Using (\ref{eq:6.2}), this gives
\[
\int_X \inner{s_k}{s_k}^m\ d\nu = \frac{c_{km}(x)}{c_k(x)^m}.
\]
But by (\ref{eq:6.3})
\begin{equation}\label{eq:6.10}
\frac{c_{km}(x)}{c_k(x)^m} \sim \Bigl(\frac{N}{2\pi}\Bigr)^{\frac{(m-1)n}{2}}\;m^{-n/2}\;
h(x)^{\frac{m-1}{2}},
\end{equation}
and we recover (\ref{eq:6.9}).
(We are grateful to Zuoqin Wang for pointing out to us this connection
between (\ref{eq:6.1})--(\ref{eq:6.3}) and these rescaling laws of
\cite{10}.)

\section{Monomials and delta functions}

In this section we present a precise way to relate the sections
$s_k$ with the Bohr-Sommerfeld fibers of the moment map.  This can
be seen as a concrete realization of the expected (or hoped-for)
equivalence between the complex polarization used in this paper
and the singular real polarization defined by such fibers.  The
result is exact (not asymptotic), so in this section $N=1$.

Let $P\subset L^*$ be the unit circle bundle, which
is a principal $S^1$ bundle with connection. 
We denote by $\calH\subset L^2(P)$ the $L^2$ closure of the space
of smooth functions that extend holomorphically to the unit
disk bundle of $L^*$, and let $\Pi: L^2(P)\to\calH$ be the
orthogonal projection.  Under the circle action $\calH$ splits into isotypical
subspaces,
\[
\calH = \widehat{\oplus}_{N}\calH_N.
\]
Specifically, $\calH_N$ consists of eigenspace of
the infinitesimal generator of the $S^1$ action in
$\calH$ corresponding to the eigenvalue $\sqrt{-1}N$.
$\calH_1$ is naturally isomorphic with the space of
holomorphic sections
of $L$.  If $s:X\to L$ is such a section, we will denote by
\[
s^\flat \in \calH_1
\]
the corresponding function on $P$.

\medskip
Since the torus $K$
acts on the bundle $L\to X$ (preserving
the hermitian structure) it acts on $P$,
preserving the connection.  The the infinitesimal
action of $K$ on $P$ is given by the Kostant formula,
(\ref{eq:3.5}), translated into this setting:
\begin{equation}\label{eq:7.1}
\forall A\in \frakk\qquad
\xi_A^\sharp = \tilde\xi_A +  H_A\,\partial_\theta.
\end{equation}
Here:
\begin{enumerate}
\item $\xi_A$ is the vector field on $X$ induced by $A$,
\item $\tilde\xi_A $ is the horizontal lift of $\xi_A$, and 
\item $H_A$ is the $A$-component of the moment map $X\to\frakk^*$,
pulled back to $P$.
\end{enumerate}
Note that, since $H_A$ is constant along trajectories
of $\tilde\xi_A$, the two fields on the right-hand side of (\ref{eq:7.1})
commute.
Furthermore, the representation of the torus $K$ on $L^2(P)$ by
translations commutes with the projection, $\Pi$. Therefore, if
$A\in\frakk$,
\[
[\L_{\xi_A^\sharp}\,,\,\Pi] = 0,
\]
where $\L$ denotes the Lie derivative.

We begin with:
\begin{lemma}
Let $k\in[\Delta]$ be a lattice point.  Then there 
exists a closed submanifold $Y_k\subset P$ such that 
\begin{enumerate}
\item $Y_k$ is horizontal, and the projection,
$P\to X$,
restricted to $Y_k$ is a diffeomorphism onto
$\phi^{-1}(k)$.
\item The restriction of $s_k^\flat$ to $Y_k$ is a
non-zero constant function.
\end{enumerate}
\end{lemma}
\begin{proof}
The inverse image $\phi^{-1}(k)$ is an orbit of $K$,
and therefore diffeomorphic to a quotient torus,
$K/K_k$, where $K_k$
is the isotropy subgroup of any point in $\phi^{-1}(k)$.
The Lie algebra of $K_k$ is the conormal space to the 
face, $F$, of $\Delta$ such 
that $k\in\mbox{Int} (F)$.
$\phi^{-1}(k)$ is an isotropic submanifold of $X$,
and so a closed horizontal lift, $Y_k$, will exist if the
holonomies of generators of the fundamental group
of $\phi^{-1}(k)$ are trivial.  
By the properties of Delzant polytopes,
we can represent generators of $\pi_1(\phi^{-1}(k))$  
by orbits of one-parameter subgroups $\exp(tA)$ 
of period $T = 2\pi$, and with $A\in\frakk$ integral.

Fix $p\in P$ above $\phi^{-1}(k)$, and note that
the curve $\exp(tA)\cdot p$ is $2\pi$ periodic
(we are denoting the action of $K$ on $P$ by a dot).
Therefore
\[
p = \exp(2\pi A)\cdot p = \exp(2\pi \tilde\xi_A) 
\circ \exp(2\pi  H_A(p)\,\partial_\theta)(p).
\]
Since $A$ and $k$ are integral, $ H_A(p) = \inner{A}{k} \in\bbZ$,
and therefore $\exp(2\pi  H_A(p)\,\partial_\theta)(p) = p$.
Therefore
\[
p = \exp(2\pi \tilde\xi_A) (p),
\]
that is, the holonomy of an orbit of $\exp(tA)$ in
$\phi^{-1}(k)$ is trivial.  This proves (1).

To prove (2), note that $\forall A\in\frakk$
the section $s_k$ satisfies
\[
\L_{\xi_A^\sharp }(s_k^\flat) =\sqrt{-1}\; \inner{A}{k}\ s_k^\flat.
\]
Taking into account that
$\L_{\partial_\theta} s_k^\flat = \sqrt{-1}\;\,s_k^\flat$, we obtain
using (\ref{eq:7.1}) that $\L_{\tilde\xi_A } (s_k^\flat) = 0$
at points over $\phi^{-1}(k)$.
Since this is true $\forall A\in \frakk$,
$s_k^\flat$ is constant on $Y_k$.  It is not zero because,
as we have seen, $\inner{s_k}{s_k}$ is in fact maximal on $\phi^{-1}(k)$.
\end{proof}

The main result of this section is:
\begin{proposition}
Let $k\in[\Delta]$ and $Y_k\subset P$ as in the previous proposition.
Let $\nu$ be the lift to $Y_k$ of a $K$-invariant density on
$\phi^{-1}(k)$.  Then the projection,
$\Pi_1(\delta_{Y_k})$, on $\calH_1$
of the resulting
delta function on $Y_k$ is a non-zero constant times $s_k$.
\end{proposition}
\begin{proof} 
Let $t_k = \Pi_1(\delta_{Y_k})$.   
We begin by clarifying that, as a distribution, $t_k$ is defined 
by the identity
\begin{equation}\label{eq:7.2}
(t_k\,,\,u) = (\delta_{Y_k}\,,\,\overline{\Pi_1(\overline{u})}),
\end{equation}
for $u$ a test function on $P$.   Therefore, if
$u\in\calH_1$,
\begin{equation}\label{eq:7.3}
\inner{t_k}{u}_{L^2} = (t_k\,,\,\overline{u}) =
\int_{Y_k}\overline{u}\,\nu.
\end{equation}
For any $A\in\frakk$ let us now compute
$\L_{\xi_A^\sharp} (t_k)$. 
Using that $[\L_{\xi_A^\sharp}\,,\,\Pi_1]=0$,
if $u$ is a test function on $P$
\[
(\L_{\xi_A^\sharp} (t_k)\,,\,u) =
-\int_{Y_k} \L_{\xi_A^\sharp} (u_1)\,\nu
\]
where $u_1 = \overline{\Pi_1(\overline{u})}$.
By (\ref{eq:7.1}) this equals
\[
-\int_{Y_k} \L_{\tilde\xi_A}(u_1)\,\nu 
+ \sqrt{-1}\,\inner{A}{k}\int_{Y_k} u_1\,\nu.
\]
The first integral is zero, because $Y_k$ is horizontal
and $\L_{\xi_A^\sharp}\nu = 0$.  The second term is
\[
\sqrt{-1}\,\inner{A}{k}\int_{Y_k}
\overline{\Pi_1(\overline{u})} \nu = 
\sqrt{-1}\,\inner{A}{k}\,(t_k\,,\,u), 
\]
using (\ref{eq:7.2}).  Therefore
$\L_{\xi_A^\sharp} (t_k) = \sqrt{-1}\,\inner{A}{k}\,t_k$, 
that is,
$t_k$ satisfies the same ODEs as $s_k^\flat$, and so
necessarily $t_k = C_k\, s_k^\flat$ for some constant $C_k$. 
To show that this constant
is not zero note that
\[
C_k = \inner{t_k}{s_k^\flat}_{L^2} = \int_{Y_k}
\overline{(s_k)^\flat}\,d\nu \not= 0
\]
by (\ref{eq:7.3}) and part (2) of the previous lemma
(in fact $C_k$ is equal to the volume of $Y_k$ times the
conjugate of the constant value of $s_k^\flat$ on $Y_k$).
\end{proof}

It is natural to ask if the analogue of the previous 
proposition holds for other integrable systems on
K\"ahler manifolds, for example the Gelfand-Cetlin
system.  We hope to return to this problem.

\end{document}